\documentclass[11pt]{amsart}
\usepackage{latexsym}
\usepackage{amssymb,amsmath,mathabx, amscd}
\usepackage[pdftex]{graphicx}
\usepackage{enumerate}
\usepackage{multicol, wrapfig}

\usepackage{colonequals}
\usepackage{comment}
\usepackage{tikz}
\usepackage{tikz-cd}
\usepackage[margin=1in]{geometry}
\usepackage[title]{appendix}

\usepackage{hyperref, mathrsfs}

\renewcommand{\aa}{\mathbb{A}}
\newcommand{\cc}{\mathbb{C}}
\renewcommand{\O}{\mathcal{O}}
\newcommand{\rr}{\mathbb{R}}
\newcommand{\pp}{\mathbb{P}}

\newcommand{\zz}{\mathbb{Z}}

\renewcommand{\gg}{\mathbb{G}}

\renewcommand{\O}{\mathcal{O}}

\newcommand{\Hom}{\operatorname{Hom}}

\renewcommand{\tilde}{\widetilde}

\newcommand{\PGL}{\text{PGL}}

\newcommand{\Tr}{\operatorname{Tr}}

\newcommand{\Sym}{\operatorname{Sym}}
\newcommand{\ind}{\operatorname{ind}}
\newcommand{\sign}{\operatorname{sign}}

\newcommand{\GW}{\operatorname{GW}}
\newcommand{\qtype}{\operatorname{Qtype}}
\newcommand{\type}{\operatorname{type}}
\newcommand{\sgn}{\operatorname{sgn}}

\renewcommand{\bar}{\overline}
\newcommand{\del}{\partial}

\newtheorem{thm}{Theorem}[section]
\newtheorem{ithm}{Theorem}
\newtheorem{lem}[thm]{Lemma}

\newtheorem{iconj}[ithm]{Conjecture}

\newtheorem{cor}[thm]{Corollary}

\theoremstyle{definition}
\newtheorem{defin}[thm]{Definition}
\newtheorem{example}[thm]{Example}
\newtheorem{algo}[thm]{Algorithm}

\newtheorem{rem}[thm]{Remark}

\usepackage{tikz}
\usepackage{tikz-cd}
\usetikzlibrary{matrix,arrows,decorations.pathmorphing}

\newcommand{\defi}[1]{\textsf{#1}} 

\title{An enriched count of the bitangents to a smooth plane quartic curve}
\author{Hannah Larson}
\author{Isabel Vogt}
\date{\today}

\begin{document}

\maketitle

\begin{abstract}
Recent work of Kass--Wickelgren gives an enriched count of the $27$ lines on a smooth cubic surface over arbitrary fields.  Their approach using $\aa^1$-enumerative geometry suggests that other classical enumerative problems should have similar enrichments, when the answer is computed as the degree of the Euler class of a relatively orientable vector bundle. Here, we consider the closely related problem of the $28$ bitangents to a smooth plane quartic. 
However, it turns out the relevant vector bundle is not relatively orientable and new ideas are needed to produce enriched counts. 
We introduce a fixed ``line at infinity," which leads to enriched counts of bitangents that depend on their geometry relative to the quartic and this distinguished line. 
\end{abstract}

\section{Introduction}
Let $k$ be a field of characteristic different from $2$.  Over $\bar{k}$, it is a beautiful and classical result of Jacobi \cite{jacobi} that for \emph{any} smooth plane quartic curve $Q \subset \pp^2_{\bar{k}}$, there exist exactly $28$ distinct lines $L \subset \pp^2_{\bar{k}}$ that are bitangent to $Q$.
The $28$ bitangent lines in $\pp^2_{\bar{k}}$ are intimately connected with the geometry of $Q$.  As $Q$ is a canonically-embedded genus $3$ curve, each bitangent gives an effective divisor $D$ on $Q$ such that $2D$ is linearly equivalent to the canonical divisor $K_Q$.  In this way, the $28$ bitangent lines correspond to the $28$ odd theta characteristics of the genus $3$ curve $Q$.  As a set, the bitangent lines in $\pp^2_{\bar{k}}$ also completely determine the curve \cite{caporaso_sernesi}. 

Over non-algebraically closed fields $k$, the situation is more subtle. For example, over $\rr$, Zeuthen proved that every smooth plane quartic has at least $4$ real bitangents \cite{salmon}, but depending upon the real topology of the quartic, it can have in total either 4, 8, 16, or 28 real bitangents.  The real bitangents play an important role in \cite{tropical}, which studies the representations of real plane quartic equations as sums of squares.

\setlength{\columnsep}{12pt}

\begin{wrapfigure}{r}{5.5cm}
\begin{center}
\includegraphics[width=5cm]{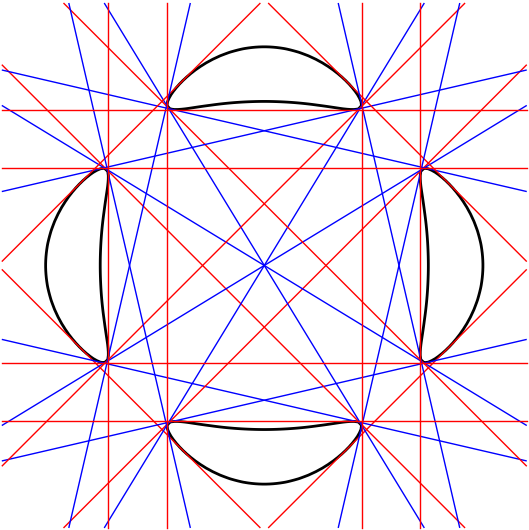} \\
The 28 real bitangents to the Trott curve colored by sign.
\end{center}
\end{wrapfigure}

\setlength{\columnsep}{5pt}

The $28$ bitangents of a smooth plane quartic are closely related to the $27$ lines on a smooth cubic surface.
Indeed, projection from a point $p$ not contained on a line on the cubic surface gives a degree $2$ map to $\pp^2$ branched over a plane quartic curve; the images of the $27$ lines and the tangent plane section at $p$ give the $28$ bitangents to this branch curve.  
Over $\rr$, a smooth cubic surface can have $3, 7, 15,$ or $27$ real lines. Segre observed that each real line can be given a sign that distinguishes its real geometry in the cubic (interpreted topologically in \cite{top});  
Finashin--Kharlamov \cite{FK} and Okonek--Teleman \cite{OT} prove that,
independent of the total number of real lines,
the corresponding signed count is always $3$. Given the intimate relationship between the $28$ bitangents to a plane quartic and these $27$ lines, it is natural to ask: Can we associate a sign to each bitangent that captures its real geometry relative to the quartic and gives rise to a constant signed count?
We present an answer to this question.

Our approach is to study the bitangents problem over arbitrary fields $k$ in the context of $\aa^1$-enumerative geometry.
Generalizing the signed count of Finashin--Kharlamov and Okonek--Teleman, Kass--Wickelgren give an enriched count of $27$ lines on a smooth cubic surface over $k$ valued in the Grothendieck-Witt group of $k$ \cite{kass_wickelgren}. This comes from an enrichment of the Euler class of a rank $4$ vector bundle on the Grassmannian $\gg(1, 3)$ of lines in $\pp^3$. Similarly, the classical count of $28$ bitangents can be found as the degree of the Euler class of a rank $4$ vector bundle on the space of lines $L$ in $\pp^2$ together with a degree $2$ subscheme $Z \subset L$.
 However, the this vector bundle is \emph{not} relatively orientable, and so the machinery of Kass--Wickelgren giving a constant enriched count breaks down. The bitangents problem therefore serves as a testing ground for using enrichment techniques on problems that are not relatively orientable. 

\vspace{.1in}
A first example of an enriched count
 is the signed count of real zeros of a polynomial over $\rr$, weighted by the sign of the derivative. 
A polynomial of even degree always has the same number of roots with positive sign as negative sign; therefore the overall signed count is always $0$.
When a polynomial has odd degree, the answer depends on the sign of the leading term of the polynomial. When it is positive, the overall signed sum is $+1$, and when it is negative, the overall signed sum is $-1$.

\begin{center}
  \begin{tikzpicture}
 \draw (0, 0) node [left] {$\rr$} -- (5, 0);
     \draw[blue] (4.3, 1.3) node{$f$};
  \draw (.7, 0) node {$\bullet$};
 \draw (1.6, 0) node {$\bullet$};
    \draw (3.8, 0) node {$\bullet$};
      \draw (.7, 0) node[above left] {\small $1$};
 \draw (1.6, 0) node[above right] {\small $-1$};
    \draw (3.8, 0) node[below right] {\small $1$};
 \draw[domain=.2:4.1,smooth,variable=\x,blue] plot ({\x},{.6*(\x-.7)*(\x-1.6)*(\x-3.8)});
 \end{tikzpicture}
 \hspace{.5in}
   \begin{tikzpicture}
 \draw (0, 0) node [left] {$\rr$} -- (5, 0);
     \draw[blue] (.6, 1.3) node{$f$};
  \draw (.7, 0) node {$\bullet$};
 \draw (1.6, 0) node {$\bullet$};
    \draw (3.8, 0) node {$\bullet$};
      \draw (.7, 0) node[below left] {\small $-1$};
 \draw (1.6, 0) node[below right] {\small $1$};
    \draw (3.8, 0) node[above right] {\small $-1$};
 \draw[domain=.2:4.1,smooth,variable=\x,blue] plot ({\x},{-.6*(\x-.7)*(\x-1.6)*(\x-3.8)});
 \end{tikzpicture}
\end{center}
The constant count in the even degree case is explained by the existence of a relative orientation for the relevant line bundle on $\pp^1$, which does not exist in the odd degree case (see Example \ref{poly}).
Nevertheless, the count in the odd degree case is constrained to a limited number of possible values, that imply, in particular, that any odd degree polynomial has at least one real root.

Our basic observation is that even if a vector bundle is not relatively orientable over the entire base, it always is away from a suitable divisor. 
This gives rise to the notion of ``relatively orientable relative to a divisor", which we explore in Section \ref{orient}.  We hope that the techniques developed there will be useful when studying other enumerative problems that lack relative orientations.  
After discussion with the authors, these ideas have already found application in the forthcoming work of McKean on enriching Bezout's Theorem when intersecting curves in $\pp^2$ whose degrees have the same parity \cite{McKean}.

For the bitangents problem, our divisor comes from a choice of line $L_\infty$ in $\pp_k^2$ and leads us to consider only quartics all of whose bitangent lines do not meet it along $L_\infty$. 
The space of such quartics is not $\aa^1$-connected, so the standard enrichment techniques do not give a constant count.  Over $\rr$, the space is not even connected in the real topology.  In Section \ref{trott} we provide examples attaining different counts over $\rr$.
Just as in the case of zeros of an odd degree polynomial, we observe that the (changing) count nevertheless contains meaningful geometric information and conjecture that it is constrained.
Finally, we relate our enriched counts of bitangents to the lines on a cubic surface by choosing $L_\infty$ to be one of the bitangents. Our construction then specializes to give a constant count of the $27$ remaining bitangents, which is equal to the Kass--Wickelgren count of lines on a cubic surface, over any ground field.

\subsection{The type of a bitangent}
When working over non-algebraically closed fields, by a \defi{line} in $\pp^2_k$ we mean a closed point $[L]$ of ${\pp_k^2}^\vee$.  We write $k(L)$ for the residue field of the point $[L]$. For an extension $K/k$, when we wish to specify a \emph{rational} point of ${\pp_K^2}^\vee$, we refer to the line as ``defined over $K$".

The type of a bitangent defined over $K$ will be an element of the Grothedieck-Witt ring $\GW(K)$. Given $a \in K^*/(K^*)^2$, we denote by $\langle a \rangle$ the equivalence class of the binary quadratic form $(x, y) \mapsto axy$. Over $\rr$, $\langle a \rangle$ depends only on the sign of $a$ and counts in $\GW(\rr)$ are the same as signed counts (assuming one knows the number of zeros over $\cc$).

Over $\rr$, the type of a bitangent relative to a fixed real line at infinity $L_\infty$ will measure the following geometric phenomenon. 
A line $L \neq L_\infty$ partitions $\aa^2_\rr = \pp^2_\rr \smallsetminus L_\infty$ into two connected components.  The affine equations for the quartic $Q$ and line $L$ allow us to choose a pair of consistent normal vectors to $Q$ at the points of bitangency with $L$.
If the two normal vectors 
lie in the same component, then $\qtype_{L_\infty}(L) = \langle 1 \rangle$. If they
 lie in different components, then $\qtype_{L_\infty}(L) = \langle -1\rangle$.
 
 \vspace{.1in}
  \begin{center}
 \begin{tikzpicture}
 \draw (0, 0) node [left] {$L$} -- (5, 0);
 \draw[red, thick, dashed] (4.5, -2)--(4.5, 2) node [above] {$L_\infty$};
  \draw (2, 0) node {$\bullet$};
 \draw (3, 0) node {$\bullet$};
 \draw[domain=1.75:3.25,smooth,variable=\x,blue] plot ({\x},{20*(\x-2)*(\x-2)*(\x-3)*(\x-3)});
 \draw [->] (2, 0) -- (2, -.75);
  \draw [->] (3, 0) -- (3, -.75);
 \draw (2.5, -2.5)  node {$\qtype_{L_\infty}(L) = \langle 1 \rangle$};
 \end{tikzpicture}
 \hspace{1in}
 \begin{tikzpicture}
 \draw (0, 0) node [left] {$L$} -- (5, 0);
 \draw [red, thick, dashed] (4.5, -2)--(4.5, 2) node [above] {$L_\infty$};
  \draw (2, 0) node {$\bullet$};
 \draw (3, 0) node {$\bullet$};
  \draw[domain=1.5:2.5,smooth,variable=\x,blue] plot ({\x},{7.7*(\x-2)*(\x-2)});
  \draw[domain=2.5:3.5,smooth,variable=\x,blue] plot ({\x},{-7.7*(\x-3)*(\x-3)});
 \draw [->] (2, 0) -- (2, -.75);
  \draw [->] (3, 0) -- (3, .75);
   \draw (2.5, -2.5)  node {$\qtype_{L_\infty}(L) = \langle -1 \rangle$};
 \end{tikzpicture}
 \end{center}
In the picture on page $1$, relative to the line $L_\infty = V(z)$, the $16$ red bitangent lines are type $\langle 1 \rangle$, 
and the $12$ blue bitangent lines are type $\langle -1 \rangle$.

A real bitangent is called \defi{split} if its points of tangency are defined over $\rr$.
Our $\qtype$ will always be $\langle 1 \rangle$ for non-split bitangents, and can be $\langle 1\rangle$ or $\langle -1 \rangle$ for split bitangents, depending on the relative geometry of the contact with the quartic and the line at infinity. 

\begin{rem}
In a different direction, Klein \cite{Klein} gives a constant signed count of flexes plus non-split bitangents:
\begin{align*}
8 &= \# \{\text{real flexes}\} + 2 \#\{\text{real non-split bitangents}\} \\
&=\#\{\text{real cusps of the dual curve}\} + 2 \#\{\text{real non-split nodes of dual curve}\}.
\end{align*}
Notice Klein's formula does not count split bitangents.
For a modern treatment see \cite{Ronga,Viro,Wall1} and \cite[Thm.~7.3.7]{Wall2}.
\end{rem}


 

 More generally,
given a line $L \subset \pp_K^2$ defined over $K$, write $\del_L$ for a derivation with respect to a linear form over $K$ vanishing along $L$; note that this is only well-defined up to multiplication by scalars in $K$. Suppose we have fixed a line $L_\infty$ defined over $k$, and we are given a homogeneous polynomial $f$ and a degree $2$ subscheme $Z = z_1+z_2 \subset L$ defined over $K$ such that $ Z \cap L_\infty = \emptyset$. By $\del_Lf(z_1)\cdot \del_L f(z_2)$ we mean to evaluate this quantity using some choice of $\partial_L$ and some choice of affine equation for $f$ on $\aa^2_{K} = \pp^2_{K} \smallsetminus L_\infty$. If the $z_i$ are defined over a quadratic extension $K'/K$, then $\del_L f(z_1)$ and $\partial_Lf(z_2)$ are elements of $K'$ that are Galois conjugate over $K$.  In any case, the product $\del_Lf(z_1)\cdot \del_L f(z_2)$ will be a well-defined element of $K/(K^*)^2$.
 
 
 
 \begin{defin} \label{qdef}
 Suppose that $f$ is a homogeneous degree $4$ polynomial in $k[y_1, y_2, y_3]$ defining a smooth plane quartic $V(f)$.  Let $L$ be a line with residue field $K$ that is bitangent to $V(f)$ with $2Z = V(f) \cap L$ and $Z \cap L_\infty = \emptyset$ for $Z = z_1 + z_2$ a degree $2$ divisor on $L$.  We define the \defi{type of the bitangent $L$ relative to $L_\infty$} to be
 \[\qtype_{L_\infty}(L) = \langle \del_Lf(z_1)\cdot \del_L f(z_2) \rangle \in \GW(K). \]
 By slight abuse of notation, given a closed point $L$ of ${\pp_k^2}^\vee$ with residue field $K$, we take $\qtype_{L_\infty}(L)$ to mean $\qtype_{L_\infty}(L')$ for $L'$ any line in the base change of $L_K$ (which is a Galois orbit of lines defined over $K$). Then $\Tr_{k(L)/k} \qtype_{L_\infty}(L)$ is a well-defined element of $\GW(k)$.
 \end{defin}

\subsection{Statement of results}
Failure of orientability manifests itself in the dependence of this type on the line at infinity. 
Nevertheless, when the ground field is $\rr$, we obtain constant signed counts for those quartics that do not meet the line at infinity over $\rr$.  The real points of such curves are compact quartics in the affine plane $\pp^2 \smallsetminus L_\infty$.

\begin{ithm} \label{cap}
Fix a line $L_\infty$ defined over $\rr$.  Let $Q$ be a smooth real plane quartic not meeting $L_\infty$ over $\rr$.  Then
\[\# \left({\text{real bitangents with}\atop \text{$\qtype_{L_\infty}(L) = \langle 1 \rangle$}} \right) - \# \left({\text{real bitangents with}\atop \text{$\qtype_{L_\infty}(L) = \langle -1 \rangle$}} \right) =4.\]
Therefore
\[\sum_{\text{lines } L \text{ bitangent to } Q} \Tr_{k(L)/k} \qtype_{L_\infty}(L) = 16\langle 1 \rangle  + 12 \langle -1 \rangle.\]
\end{ithm}

\begin{rem}
Theorem \ref{cap} immediately implies that if the real points of $Q$ form a compact quartic in an affine plane $\rr^2$, then $Q$ has at least $4$ real bitangent lines.  By \cite{Reich}, such an affine plane exists for every smooth plane quartic; however, this proof presuposes the existence of one real bitangent. 
\end{rem}

%

We will see in Section \ref{vary} that moving the line at infinity can change the signed count.  Code available at \cite{code} computes all signed counts that can be realized for a fixed quartic by varying $L_\infty$.
Based on a randomized search of over $10000$ quartics --- using code from \cite{tropical} to generate quartics of each topological type --- we conjecture that the signed count is constrained to a certain range.

\begin{iconj}\label{conj_range}
Let $Q$ be a smooth plane quartic defined over $\rr$, and let $L_\infty \subset \pp^2_{\rr}$ be a line defined over $\rr$ such that $L \cap Q \cap L_\infty = \varnothing$ for all bitangents $L$. Then
\[\# \left({\text{real bitangents with}\atop \text{$\qtype_{L_\infty}(L) = \langle 1 \rangle$}} \right) - \# \left({\text{real bitangents with}\atop \text{$\qtype_{L_\infty}(L) = \langle -1 \rangle$}} \right) \in \{0,2,4,6,8\}.\]
\end{iconj}


Finally, over any ground field $k$, if the line at infinity is chosen to be \emph{one of the bitangents}, then the remaining $27$ have a constant enriched count relative to the distinguished one.

\begin{ithm} \label{main}
Let $Q$ be a smooth plane quartic defined over $k$ and let $L_\infty$ be a bitangent to $Q$ defined over $k$. Then
\[\sum_{\substack{\text{lines } L \text{ bitangent to } Q \\ L \neq L_\infty}} \Tr_{k(L)/k} \qtype_{L_\infty}(L) = 15\langle 1 \rangle + 12\langle -1 \rangle \in \GW(k).\]
\end{ithm}

\begin{rem}
Both theorems apply over $\rr$ when $L_\infty$ is a non-split bitangent. In this case, taking the trace in Theorem \ref{main} gives a signed count of $3$ for bitangents other than $L_\infty$. Non-spit bitangents always have type $\langle 1 \rangle$, so if $L_\infty$ is counted too then we recover Theorem \ref{cap}.
\end{rem}

\begin{rem}
During the preparation of this paper, V.~Kharlamov, R.~Rasdeaconu, and S.~Finashin informed us of a related signed count.  
Instead of real bitangents to a smooth plane quartic, they consider real lines on a real del Pezzo surface $Y$ that is the double cover of the projective plane branched over the quartic, so that each real bitangent is replaced by two such lines. Their signed count of real lines on $Y$ uses an appropriate $\mathrm{Pin}^-$ structure and, for instance, attributes opposite signs to two real lines covering the same real bitangent. Thus, the signed count of all the real lines on $Y$ is zero. However, the partial sum over all the real lines intersecting a fixed line with odd multiplicity (including itself) is equal to $\pm 4$ (which gives another explanation of the existence of at least 4 real bitangents to a smooth plane quartic).
\end{rem}

 \subsection*{Acknowledgements}
Thanks to Jesse Kass and Kirsten Wickelgren for many insightful conversations, comments on several drafts of this article, and for advising the $\aa^1$-enumerative geometry problem session at the the 2019 Arizona Winter School. We are grateful to the organizers, funders, and other participants --- in particular Ethan Cotterill, Ignacio Darago, and Changho Han --- of the Winter School for fostering the stimulating environment that inspired this work.

\section{Relative orientability relative to a divisor} \label{orient}
Many classical enumerative questions are solved by counting the zeros of sections of a vector bundle on a projective variety.
If a section $\sigma$ of a rank $n$ vector bundle $E$ on an $n$-dimensional smooth projective variety $X$ has isolated zeros, the degree of the top Chern class or Euler class $c_n(E) \in H^{2n}(X)$ gives the number of zeros of $\sigma$ over the algebraic closure, counted with multiplicity. 
Over non-algebraically closed fields, the number of zeros of a section need not be constant. Recall that a manifold $X$ is \defi{orientable} if $\det T_X \cong \O_X$ and a choice of isomorphism is called an \defi{orientation}. 
Given a vector field (i.e. a global section of $T_X$) with isolated zeros on an orientable real manifold, one can obtain a constant signed count of zeros using local indices, as defined by Milnor in \cite{Milnor}. 
 If $p$ is a simple zero of $\sigma$, the local index is computed as follows: On an open neighborhood $U \ni p$ where $(T_X)|_U \cong \rr^{\oplus n}$, the vector field is represented by $n$ functions $(\sigma_1,\ldots, \sigma_n)$. The local index is the sign of the Jacobian determinant $\sign J_p(\sigma)$. 
It turns out that the sum of these local indices
is independent of the vector field, giving the first example of an ``enriched count."

The above has been generalized to sections of relatively orientable vector bundles on projective varieites over arbitrary fields, see work of Kass--Wickelgren \cite{kass_wickelgren} and references therein. 
A vector bundle $E$ is said to be \defi{relatively orientable} if $\Hom(\det T_X,  \det E) \cong L^{\otimes 2}$ for some line bundle $L$ and the choice of such an isomorphism is called a \defi{relative orientation}. 
They define an \defi{enriched Euler class} as a sum of local indices  valued in the Grothendieck-Witt group of the ground field
\[e(E, \sigma) \colonequals \sum_{p : \sigma(p) = 0}\ind_p(\sigma) \in \GW(k), \]
and show that this class in $\GW(k)$ is constant on $\aa^1$-connected components of the space of sections.

The rank provides an isomorphism $\GW(\cc) \cong \zz$; the rank and the signature induce an isomorphism $\GW(\rr) \cong \zz \oplus \zz$.
Given $a \in k^*$, let $\langle a \rangle \in \GW(k)$ denote the class of the rank $1$ bilinear form $(x, y) \mapsto axy$. 
Thus over $\rr$, $\langle a \rangle$ is the same as the information of the sign of $a$.  For simple zeros of a real section, the Kass--Wickelgren local index is $\langle J_p(\sigma) \rangle$, recovering Milnor's local index.  
When the rank is known, the enriched Euler class is therefore determined by the \defi{signed count}
\[s(E, \sigma) \colonequals \sum_{p \in X(\rr) \atop \sigma(p) = 0} \sgn \ind_p(\sigma) \in \zz. \]

The following example demonstrates the necessity of relative orientability for obtaining constant enriched counts and suggests what we may study instead without it.
\begin{example} \label{poly}
Consider the line bundle $E = \O_{\pp^1}(d)$ on $\pp_\rr^1$. We have $T_{\pp^1} \cong \O_{\pp^1}(2)$, and so $E$ is relatively orientable if and only if $d$ is even. 
Global sections of $E$ correspond to homogeneous degree $d$ polynomials on $\pp^1$. 
A relative orientation supplies an oriented coordinate $t$ in an affine patch around any zero, and for simple zeros,
the local index measures the sign of the derivative.
If $d$ is even, for any section $f$, we have

\begin{multicols}{2}
\[s(\O_{\pp^1}(d), f) = 0\]
\columnbreak
\begin{center}
 \begin{tikzpicture}
 \draw (0, 0) node [left] {$\pp^1$} -- (5, 0);
 \draw[red] (4.5, 0) node {$\bullet$};
  \draw[red] (4.5, -.5) node {$\infty$};
  \draw[blue] (4.2, 1.3) node{$f$};
  \draw (.7, 0) node {$\bullet$};
 \draw (1.6, 0) node {$\bullet$};
  \draw (2.6, 0) node {$\bullet$};
    \draw (3.8, 0) node {$\bullet$};
      \draw (.7, -1) node {\small $\langle-1 \rangle$};
 \draw (1.6, -1) node {\small $\langle 1 \rangle$};
  \draw (2.6, -1) node {\small $\langle -1 \rangle$};
    \draw (3.8, -1) node {\small $\langle 1 \rangle$};
 \draw[domain=.5:3.95,smooth,variable=\x,blue] plot ({\x},{(\x-.7)*(\x-1.6)*(\x-2.6)*(\x-3.8)});
 \end{tikzpicture}
 \end{center}
\end{multicols}

If $d$ is odd, we might still try to naively sum local indices of zeros of a section, and it will make sense in an affine $\aa_\rr^1 \subset \pp_\rr^1$. With respect to a coordinate $t$ on $\aa_\rr^1$, we may write $f = a_d t^d + \ldots + a_1 t + a_0$. 
Then we find
\[s(\O_{\pp^1}(d), f) = \begin{cases} 1 & \text{if $a_d > 0$} \\
-1 &\text{if $a_d < 0$.}\end{cases}\]

\vspace{.15in}
\begin{center}
  \begin{tikzpicture}
 \draw (0, 0) node [left] {$\pp^1$} -- (5, 0);
 \draw[red] (4.5, 0) node {$\bullet$};
   \draw[red] (4.5, -.5) node {$\infty$};
     \draw[blue] (4.3, 1.3) node{$f$};
  \draw (.7, 0) node {$\bullet$};
 \draw (1.6, 0) node {$\bullet$};
    \draw (3.8, 0) node {$\bullet$};
      \draw (.7, -1) node {\small $\langle 1 \rangle$};
 \draw (1.6, -1) node {\small $\langle -1 \rangle$};
    \draw (3.8, -1) node {\small $\langle 1 \rangle$};
 \draw[domain=.2:4.1,smooth,variable=\x,blue] plot ({\x},{.6*(\x-.7)*(\x-1.6)*(\x-3.8)});
 \end{tikzpicture}
 \hspace{.5in}
   \begin{tikzpicture}
 \draw (0, 0) node [left] {$\pp^1$} -- (5, 0);
 \draw[red] (4.5, 0) node {$\bullet$};
   \draw[red] (4.5, -.5) node {$\infty$};
     \draw[blue] (.6, 1.3) node{$f$};
  \draw (.7, 0) node {$\bullet$};
 \draw (1.6, 0) node {$\bullet$};
    \draw (3.8, 0) node {$\bullet$};
      \draw (.7, -1) node {\small $\langle -1 \rangle$};
 \draw (1.6, -1) node {\small $\langle 1 \rangle$};
    \draw (3.8, -1) node {\small $\langle -1 \rangle$};
 \draw[domain=.2:4.1,smooth,variable=\x,blue] plot ({\x},{-.6*(\x-.7)*(\x-1.6)*(\x-3.8)});
 \end{tikzpicture}
\end{center}
\vspace{.1in}
In other words, we obtain an enriched count of the $d$ zeros when $a_d \neq 0$, but it depends on $f$. 
Moreover, to make this enrichment, we had to choose a divisor $\infty \in \pp^1$. We then considered only sections that do not vanish along this divisor and found that there are two regions --- corresponding to positive and negative leading coefficient --- where different signed counts are attained.

An alternative approach would be to choose $\infty$ so that $f$ has a simple zero at $\infty$: then the signed count of the remaining zeros is constant.
\end{example}

The above example suggests that, even for non-orientable problems, we can make geometric meaning of local indices away from a suitably chosen divisor. 
\begin{defin} \label{relo}
We say that a vector bundle $E$ on a smooth projective variety $X$ is \emph{relatively orientable relative to an (effective) divisor} $D$ if $\Hom(\det T_X,  \det E) \otimes \O(D) \cong L^{\otimes 2}$ for some line bundle $L$. Equivalently, $E$ is relatively orientable on the open subvariety $X \smallsetminus D$.
\end{defin}
\begin{rem}
Every vector bundle is relatively orientable relative to a divisor, and there may be many choices of a divisor.
In practice, one should select an effective divisor that is geometrically meaningful in some way.
\end{rem}

Over $\rr$, Definition \ref{relo} allows us to make precise the phenomenon observed in Example \ref{poly}. The definition of local index by Milnor and its generalization by Kass--Wickelgren relies only on compatible trivializations over open neighborhoods of zeros. Thus, the local index of a section of a relatively orientable vector bundle on $X \smallsetminus D$ is well-defined at any isolated zero not in $D$.

Given a divisor $D$, we denote by $V_D \subset H^0(E)$ the locus of real sections with a real zero along $D$.
The following lemma extends enrichment techniques over $\rr$ to a broader setting, at the expense of removing those sections in $V_D$.

\begin{lem} \label{ouridea}
Let $X$ be a smooth real projective variety.
Suppose $E$ is relatively oriented relative to an effective divisor $D \subset X$.  
Let $H^0(E)^\circ$ denote the space of sections with isolated zeros. Then 
$s(E, \sigma)$, and hence $e(E, \sigma)$, is constant for $\sigma$ in any connected component of $H^0(E)^{\circ} \smallsetminus V_D$.
\end{lem}
\begin{proof}
Because the rank of $e(E, \sigma)$ is constant for algebraic sections, it suffices to show $s(E, \sigma)$ is constant on connected components of $H^0(E)^{\circ} \smallsetminus V_D$. More generally, the signed count is continuous on the subset $A \subset C^\infty(X, E)$ of $C^\infty$ sections of the real vector bundle $E$ with isolated zeros that are not contained in $D$.

Let $A' \subset A$ denote the subspace of sections which have simple zeros.
Suppose $p$ is a simple zero of a section $\sigma \in A'$.
Let $p \in U \subset X \smallsetminus D$ be an open neighborhood and choose isomorphisms $\phi: E|_U \cong \rr^{\oplus n}$ and $\psi: T_U \cong \rr^{\oplus n}$ such that 
$\det \phi^{-1} \circ \psi \in  \Hom(\det (T_X)(U),  \det E(U)) \cong L(U)^{\otimes 2}$ is a square under the relative orientation on $X \smallsetminus D$. With respect to these trivializations, $\sigma$ is represented by $n$ functions $(\sigma_1, \ldots, \sigma_n)$ and $\ind_p\sigma = \langle \det J_\sigma(p) \rangle$. Because the Jacobian is continuous, and $\rr^* \rightarrow \GW(\rr)$ by $a\mapsto \langle a \rangle$ is continuous, it follows that $s(E, \sigma)$ is continuous on $A'$.

Via the composition $\psi^{-1} \circ \phi$, each section $\sigma$ of $E|_U$ gives us a vector field $v$ on $U$. We now apply Milnor's local alteration as in \cite[``Step 2" of \textsection 6]{Milnor}. Suppose $p$ is a non-simple zero. Let $p \in N_1 \subset N \subset U$ be sufficiently small nested neighborhoods (in the real topology) and let $\lambda : U \rightarrow [0, 1]$ be a smooth function such that $\lambda(x) = 1$ for $x \in N_1$ and $\lambda(x) = 0$ for $x$ outside $N$. If $y$ is a sufficiently small regular value of $v$, then $v'(x) = v(x) - \lambda(x) y$ defines a vector field which is non-degenerate within $N$. By \cite[\textsection 6, Thm. 1]{Milnor}, the sum of Milnor's local indices at the zeros within $N$ is the degree of the ``Gauss mapping" $\overline{v} : \partial N \rightarrow S^{m-1}$, and hence does not change during this alteration. Applying this alteration locally around each non-simple zero shows that $s(E, \sigma)$ is continuous on $A$.
\end{proof}

\begin{rem}
Working over an arbitrary ground field, a
natural replacement for $V_D$ is the algebraic hypersurface $\tilde{V}_D \subset H^0(E)$ of sections vanishing at a closed point of $D$. In general, $H^0(E)^\circ \smallsetminus \tilde{V}_D$ has no non-trivial $\aa^1$-connected components, so the results of Kass--Wickelgren do not apply to give constant enriched counts. Thus, when working over $\rr$, Lemma \ref{ouridea} is stronger than the Kass--Wickelgren machinery. However, we believe this additional strength is special to $\rr$ and does not generalize readily to other fields. Notice also that over $\rr$, $V_D$ is contained in the real points of $\tilde{V}_D$ but need not equal it. In other words, Lemma \ref{ouridea} provides signed counts even when $\sigma$ has a pair of complex conjugate zeros along $D$. 

\end{rem}

Lemma \ref{ouridea} suggests the following approach to enriching non-orientable problems over $\rr$. 
First, restrict attention to sections with zeros away from a suitably chosen divisor.
The local index may then have a geometrically meaningful interpretation relative to this divisor.
The locus $V_D \subset H^0(E)$ will be codimension $1$, so we expect the complement $H^0(E) \smallsetminus V_D$ to have many components. 
However, on each component of the complement, the signed count is constant. The locus $V_D \subset H^0(E)$  should be thought of as ``walls" in the space of sections, where the relative orientation cannot be extended consistently. Signed counts change as one moves across these walls. One might then try to characterize the different components of $H^0(E) \smallsetminus V_D$ and thereby all of the possible signed counts. 

In the remaining sections, we carry out this procedure for the problem of $28$ bitangents. We characterize a natural connected region where the signed count is constant and
give examples demonstrating different possible signed counts. We also conjecture a list of all signed counts that are realized. Finally, we relate enriched counting of bitangents to the enriched count of $27$ lines on a cubic surface, in a manner akin to allowing one of the zeros in Example \ref{poly} to be at $\infty$. This last method yields results over arbitrary fields.


\section{The local index for bitangents}

In this section we will define a space $X$ of dimension $4$ and a bundle $E$ on $X$ of rank $4$, such that every quartic equation $f \in H^0( \pp^2, \O_{\pp^2}(4))$ gives rise to a section $\sigma_f$ of $E$ whose zeros correspond to the bitangents of $V(f)$.
A choice of a line $L_{\infty}$ in the plane determines a divisor $D_\infty \subset X$ such that $E$ is relatively orientable relative to $D_\infty$.
We compute the local index $\ind_p \sigma_f$ with respect to a relative orientation on $X \smallsetminus D_\infty$, and show that it agrees with our geometric definition of $\qtype$.

Let $S$ denote the (rank 2) tautological bundle on the Grassmannian ${\pp^2}^\vee$ of lines in $\pp^2$.  The $\overline{k}$-points of the projective bundle 
\[X  \colonequals \pp \Sym^2 S^\vee \]
correspond to the pairs $(L, Z)$, where $L \subset \pp^2$ is a line and $Z \subset L$ is a degree $2$ subscheme.  Write $\pi \colon X \to {\pp^2}^\vee$ for the natural projection map.  As a projective bundle, intersection theory on $X$ is straightforward, and we recall the key facts here.  The Picard group of $X$ is generated by pullbacks from ${\pp^2}^\vee$
and by a (relative) hyperplane class $\O_{X}(1)$.  The fiber of the bundle $\pi^*\Sym^2S^\vee$ at a point $(L, Z)$ is the $3$-dimensional space of quadratic polynomials on the line $L$ and the fiber of the universal subbundle $\O_X(-1) \hookrightarrow \pi^* \Sym^2S^\vee$ is the $1$-dimensional space spanned by a choice of quadratic equation defining $Z \subset L$.

The fiber of our vector bundle $E$ at a point $(L, Z)$ will be isomorphic to the space of quartic polynomials on $L$ modulo the square of an equation of $Z$.  Precisely, we define $E$ to be the quotient of $\Sym^4S^\vee$ by the subbundle $\O_X(-2)$ including via the tensor product
\[\O_X(-2) \simeq \O_X(-1) \otimes \O_X(-1) \to \Sym^2S^\vee \otimes \Sym^2 S^\vee \to \Sym^4 S^\vee. \]
Alternatively, recall that any degree $4$ subscheme $\Gamma \subset \pp^1$ imposes independent conditions on polynomials of degree $4$, and the kernel of the restriction map
\[H^0(\pp^1, \O_{\pp^1}(4)) \to H^0(\Gamma, \O_{\pp^1}(4)|_\Gamma) \]
is precisely the $1$-dimensional space spanned by an equation of $\Gamma$.  In particular, we can apply this to a degree $4$ subscheme $\Gamma = 2Z$, the square of a degree $2$ subscheme $Z \subset L$; the kernel of the restriction map is now generated by the square of the equation of $Z$.  Over the parameter space of such $(L,Z)$, we may therefore view the map $\Sym^4S^\vee \to E$ in the presentation
\[0 \to \O_X(-2) \to \Sym^4 S^\vee \to E \to 0\]
as evaluation of a quartic polynomial along $2Z \subset L$.

Given a quartic equation $f \in H^0(\pp^2, \O_{\pp^2}(4))$, restriction of $f$ to any line $L \subset \pp^2$ defines a section $\sigma_f$ of $\Sym^4 S^\vee$, and hence of $E$.  Alternatively, the section $\sigma_f$ of $E$ at $(L, Z)$ takes the evaluation of $f$ along $2Z$.  Evidently, the section $\sigma_f$ vanishes at $(L, Z)$ if and only if $L$ is a bitangent of the quartic plane curve $V(f)$ with points of tangency at $Z$ in $L$.
Using the splitting principle, it is not hard to show that $\deg c_4(E) = 28$, recovering the classical count.

By the splitting principle we have
\[\det E = \det \Sym^4 S^\vee \otimes \O_X(2) = (\det S^\vee)^{\otimes 10} \otimes \O_X(2) = \pi^* \O_{{\pp^2}^\vee}(10) \otimes \O_X(2). \]
On the other hand, using the relative Euler sequence 
\[0 \to \O_X \to \pi^*\Sym^2 S^\vee \otimes \O_X(1) \to T_{X/{\pp^2}^\vee} \to 0\]
for our projective bundle $X$, we compute
\begin{align*}
\det T_X &= \pi^*\det T_{{\pp^2}^\vee} \otimes \det T_{X/{\pp^2}^\vee} = \pi^*\det T_{{\pp^2}^\vee} \otimes \det (\pi^*\Sym^2 S^\vee \otimes \O_X(1)) \\
&=\pi^* \O_{{\pp^2}^\vee}(3) \otimes \det \pi^*\Sym^2 S^\vee \otimes \O_X(3) = \pi^*\O_{{\pp^2}^\vee}(6) \otimes\O_X(3).
\end{align*}
In particular,
\begin{equation} \label{rel}
\Hom(\det T_X, \det E) = \pi^*\O_{{\pp^2}^\vee}(4) \otimes \O_X(-1),
\end{equation}
which is not a tensor square, and the obstruction is the copy of $\O_X(-1)$. 

We now describe a section of $\O_X(1)$ defining an effective divisor $D_\infty$, away from which $E$ is relatively orientable.
Sections of $\O_X(1)$ restricted to the fiber over $L \in {\pp^2}^\vee$ correspond to linear forms on the space of quadratic polynomials on $L$.
 Given a line $L_\infty$, for each $L \neq L_\infty$, evaluation of quadratic polynomials at the point $L \cap L_\infty$ defines a section of $\O_X(1)|_{\pi^{-1}(L)}$. 
Together, this defines a section of $\O_X(1)$ away from $\pi^{-1}(L_\infty)$, which is codimension $2$. Let $D_\infty = \{(L, Z) : Z \cap L_\infty \neq \emptyset\}$ be the closure in $X$ of the vanishing locus of this section. As $X$ is smooth and $\O_X(1)|_{X \smallsetminus \pi^{-1}(L_\infty)} \cong \O(D_\infty)|_{X \smallsetminus \pi^{-1}(L_\infty)}$ these line bundles are isomorphic on all of $X$. 
Equation \eqref{rel} shows that $E$ is relatively orientable on the complement of $D_\infty$.
Thus, we can give $E|_{X \smallsetminus D_\infty}$ a relative orientation and make sense of the local index $\ind_{(L, Z)} \sigma_f$ at a zero $(L,Z) \notin D_\infty$. 

\begin{center}
\begin{tikzpicture}
\draw [->] (2.5, 5.5) node [above] {$E$} -- (2.5, 4.8);
\draw (0, 2) .. controls (1,2.2) .. (2.5, 2);
\draw (2.5, 2) .. controls (4,1.8) .. (5, 2);
\draw (0, 2+2) .. controls (1,2.2+2) .. (2.5, 2+2);
\draw (2.5, 2+2) .. controls (4,1.8+2) .. (5, 2+2);
\draw (0, 4) -- (.5, 4.5);
\draw (5, 4) -- (5.5, 4.5);
\draw (5, 2) -- (5.5, 2.5) -- (5.5, 4.5);
\draw (0+.5, 2+2+.5) .. controls (1+.5,2.2+2+.5) .. (2.5+.5, 2+2+.5);
\draw (2.5+.5, 2+2+.5) .. controls (4+.5,1.8+2+.5) .. (5+.5, 2+2+.5);
\draw (0, 2) -- (0, 4);
\draw (5, 2) -- (5, 4);
\draw (6, 3) node {$X$};
\draw (6, .5) node {${\pp^2}^\vee$};
\draw (3, -.25) -- (0, 0) -- (2, 1);
\draw (2, 1) -- (2+3, 1-.25) -- (3, -.25);
\draw [red] (0, 2+1) .. controls (1,2.2+1) .. (2.5, 2+1);
\draw [red] (2.5, 2+1) .. controls (4,1.8+1) .. (5, 2+1);
\draw [red] (0+.5, 2+1+.5) .. controls (1+.5,2.2+1+.5) .. (2.5+.5, 2+1+.5);
\draw [red] (2.5+.5, 2+1+.5) .. controls (4+.5,1.8+1+.5) .. (5+.5, 2+1+.5);
\draw [red] (5+.5, 2+1+.5) -- (5, 2+1);
\draw [red] (0, 2+1) -- (0+.5, 2+1+.5);
\draw [red] (2+.25, .2+.25) node {$\bullet$};
\draw [red] (2+.25, .2+.25) node [right] {$L_\infty$};
\draw [red] (2+.25, 2.08+.25) -- (2+.25, 4.08+.25);
\draw [red] (0, 3) node [left] {$D_\infty$};
\draw[<->] (3+9, -.25) -- (9, 0) -- (2+9, 1);
\draw (3+9.2, -.25) node {$b$};
\draw (2+9.2, 1.1) node {$a$};
\draw [->] (9, 2) -- (9.5,2.5);
\draw (9.6, 2.6) node {$s$}; 
\draw [->] (9, 2) -- (9, 4) node[above] {$t$};
\draw (11, 5.8) node {$y_2^3y_3, y_2^2y_3^2, y_2y_3^3, y_3^4$};
\end{tikzpicture}
\end{center}

\vspace{.1in}

We may choose coordinates $[y_1, y_2, y_3]$ on $\pp_K^2$ so that 
\begin{equation} \label{good}
L_\infty = V(y_3), \qquad L = V(y_1), \qquad \text{and} \qquad Z = V(y_2^2+ \alpha y_3^2),
\end{equation}
where $\alpha \in K$.
(Our assumption $Z \cap L_\infty = \varnothing$ means the coefficient of $y_2^2$ in the defining equation of $Z$ is non-zero, so we may always complete the square.)
Given coordinates $[y_1, y_2, y_3]$ on $\pp^2$ such that \eqref{good} holds, we now describe a procedure for giving ``standard affine coordinates" around each $K$-point $(L, Z) \notin D_\infty$. 
Let $\aa^2 \subset (\pp^2)^\vee$ be the affine patch with coordinates $(a, b)$ corresponding to the line $L_{a,b} = V(y_1 + a y_2 + by_3)$, so $L = L_{0,0}$ is the origin. On this affine, the vector bundle $S^\vee$ is trivialized by $y_2$ and $y_3$ (by which mean the sections obtained by restricting the linear forms $y_2$ and $y_3$ to each of the lines). Thus, $\pi^{-1}(\aa^2) \subset X$ is identified with $\aa^2 \times \pp \langle y_2^2, y_2y_3, y_3^2 \rangle$, where $\langle y_2^2, y_2y_3, y_3^2 \rangle$ denotes the three-dimensional vector space spanned by $y_2^2$, $y_2y_3$, and $y_3^2$. Let $\aa^2 \subset \pp \langle y_2^2, y_2y_3, y_3^2 \rangle$ be the affine plane with coordinates $(s, t)$ corresponding to $Z_{s,t} = V((y_2^2 + \alpha y_3^2) + s y_2y_3 + t y_3^2)$, so $Z = Z_{0,0}$ is the origin $(s, t) = (0, 0)$ here. 
We refer to such an $\aa^2 \times \aa^2 \cong \aa^4$ with coordinates $(a, b, s, t)$ as ``standard affine coordinates centered at $(L, Z)$."

To each choice of coordinates as above, we associate a trivialization of $E$.
Corresponding to our trivialization of $S^\vee$, the vector bundle $\mathrm{Sym}^4 S^\vee$ is trivialized by $y_2^4$, $y_2^3y_3$, $y_2^2y_3^2$, $y_2y_3^3$, and $y_3^4$. Over our standard affine chart, the tautological bundle $\O_X(-1)$ is trivialized by the non-vanishing section $(y_2^2 + \alpha y_3^2) + s y_2y_3 + t y_3^2$, and $\O_X(-2)$ is trivialized by its square $((y_2^2 + \alpha y_3^2) + s y_2y_3 + t y_3^2)^2$. Using the relation 
\begin{equation} \label{y2rel}
y_2^4 = -(2s y_2^3y_3 + 2(\alpha + t) y_2^2y_3^2 + 2s\alpha y_2y_3^2 + (\alpha^2+2\alpha t)y_3^4) +O((s,t)^2),
\end{equation}
the monomials $y_2^3y_3$, $y_2^2y_3^2$, $y_2y_3^3$, and $y_3^4$ trivialize the quotient bundle $E$ over our standard affine chart.

The following lemma will allow us to use these nice coordinates to compute local indices.

\begin{lem} \label{exists}
There exists a relative orientation $\Hom(\det T_{X \smallsetminus D_\infty},\det E|_{X \smallsetminus D_\infty}) \cong \pi^*\O_{{\pp^2}^\vee}(2)|_{X \smallsetminus D_\infty}^{\otimes 2}$ on $X \smallsetminus D_\infty$ such that
for every standard affine chart $U \cong \aa^4_{(a, b, s, t)}$,
 the map $\det (T_X)|_U \rightarrow \det E|_U$ induced by sending the basis $(da, db, ds, dt)$ to the associated trivialization $(y_2^3y_3, y_2^2y_3^2, y_2y_3^3,y_3^4)$ is a tensor square of an element of $H^0(U, \pi^*\O_{{\pp^2}^\vee}(2)|_U)$.
\end{lem}
\begin{proof}
Suppose that $U$ is a standard affine associated to coordinates $[y_1, y_2, y_3]$ on $\pp^2$. It suffices to show that the orientation coming from
\[ (da, db, ds, dt)\mapsto (y_2^3y_3, y_2^2y_3^2, y_2y_3^3,y_3^4)\]
 agrees with the orientation induced by any other choice of coordinates $[y_1', y_2', y_3']$ satisfying \eqref{good} for any $L = L_{a_0,b_0}$ and $Z = Z_{s_0,t_0}$ with $(a_0,b_0, s_0, t_0)  \in U$. By convention, $y_3$ and $y_3'$ vanish along $L_\infty$, so $y_3'$ is a multiple of $y_3$.
Furthermore, the equation of $Z_{s_0,t_0}$ restricted to $V(y_1')$ has no cross term with respect to $y_2', y_3'$. Thus, $U'$ corresponds to
\begin{align*}
y_1' &= \lambda_1(y_1 + a_0y_2 + b_0y_3) \\
y_2' &= \lambda_2 \left(y_2 + \frac{s_0}{2} y_3\right) + \mu y_1' \\
y_3' &= \lambda_3 y_3
\end{align*}
for some $\lambda_i \in K^\times$ and $\mu \in K$. 
To determine the change of basis matrix for $(da', db', ds', dt')$ to $(da, db, ds, dt)$
we write
\begin{align*}
y_1' + a'y_2' + b'y_3' &=  (\lambda_1+a'\mu)(y_1 + a_0y_2 + b_0y_3) + a'\lambda_2 \left(y_2 + \frac{s_0}{2} y_3\right) + b' \lambda_3 y_3 \\
&= (\lambda_1+a'\mu)\left(y_1 + \left(a_0+\frac{a'\lambda_2}{\lambda_1 + a'\mu}\right) y_2 + \left(b_0 + \frac{a'\lambda_2 s_0+ 2b'\lambda_3}{2(\lambda_1 + a'\mu)} \right)y_3\right),
\end{align*}
to see that at $(a',b') =(0,0)$, we have
\begin{align*}
da = \frac{\lambda_2}{\lambda_1} da' \qquad \text{and} \qquad db = \frac{\lambda_2 s_0}{2 \lambda_1} da' + \frac{\lambda_3}{\lambda_1}db'.
\end{align*}
Similarly, the equation for $Z'_{s',t'}$ on $V(y_1')$ is
\begin{align*}
y_2'^2 + s' y_2'y_3' + (\alpha' + t') y_3'^2 &= \lambda_2^2\left( y_2^2 + s_0y_2y_3 + \frac{s_0^2}{4}y_3^2  + s'\frac{\lambda_3}{\lambda_2} y_2y_3 + s'\frac{s_0\lambda_3}{2\lambda_2} y_3^2 + (\alpha' + t') \frac{\lambda_3^2}{\lambda_2^2} y_3^2\right) \\
&=  \lambda_2^2\left( y_2^2  + \left(s_0 + s'\frac{\lambda_3}{\lambda_2}\right) y_2y_3 + \left(\frac{s_0^2}{4} +\alpha' \frac{\lambda_3^2}{\lambda_2^2} + s'\frac{s_0\lambda_3}{2\lambda_2} + t' \frac{\lambda_3^2}{\lambda_2^2} \right)y_3^2\right),
\end{align*}
which shows that at $(s, t) = (0, 0)$, we have
\[ds = \frac{\lambda_3}{\lambda_2} ds' \qquad \text{and} \qquad dt = \frac{s_0\lambda_3}{2\lambda_2} ds' + \frac{\lambda_3^2}{\lambda_2^2} dt'.\]
Thus, the change of basis for $(da', db', ds', dt')$ to $(da, db, ds, dt)$ has determinant $\frac{\lambda_3^4}{\lambda_1^2\lambda_2^2}$.
Since $y_3'$ is a multiple of $y_3$, the change of basis for $(y_2'^3y_3', y_2'^2y_3'^2, y_2'y_3'^3,y_3'^4)$ to
$(y_2^3y_3, y_2^2y_3^2, y_2y_3^3,y_3^4)$ is upper-triangular. The product of diagonal entries is $(\lambda_2^3 \lambda_3)(\lambda_2^2\lambda_3)(\lambda_2\lambda_3^3)(\lambda_3^4) = \lambda_2^6\lambda_3^{10}$. Because both change of basis matrices have square determinants, the two possible relative orientations agree.
\end{proof}

We now show that the local index encodes the $\qtype$ of bitangents, as defined in Definition \ref{qdef}.
 

\begin{lem}  \label{indlem}
 Let $L$ be a line defined over $K$.
Let $f$ be a smooth quartic over $K$ such that $\sigma_f$ has an isolated zero at $(L, Z = z_1 + z_2)$ and $Z \cap L_\infty = \emptyset$.  Then
\begin{equation}\label{local_index} \ind_{(L, Z)} \sigma_f = \qtype_{L_\infty}(L) \ \text{in } \GW(K). \end{equation}
\end{lem}
\begin{proof}
If $(L, Z)$ is a zero of $\sigma_f$, then working in standard affine coordinates centered at $(L, Z)$, we have
 $f|_L = (y_2^2 + \alpha y_3^2)^2$ for some $\alpha \in K$. 
In particular, $f$ is of the form
\[f(y_1, y_2, y_3) = (y_2^2 + \alpha y_3^2)^2 + y_1 ( c_{1,3,0} y_2^3 + c_{1,2,1}y_2^2 y_3 + c_{1,1,2} y_2y_3^2 + c_{1,0,3}y_3^3) + O(y_1^2). \]

To evaluate the right hand side of \eqref{local_index}, we work in the affine patch $y_3 = 1$, wherein $z_1 = [0, d, 1]$ and $z_2 = [0, -d, 1]$ for $d^2 = -\alpha$.  Up to squares, we have
\begin{align}
\partial_L f(z_1) \cdot \del_L f(z_2) &= (c_{1,3,0} d^3 + c_{1,2,1}d^2 + c_{1,1,2} d + c_{1,0,3})(-c_{1,3,0} d^3 + c_{1,2,1}d^2 - c_{1,1,2} d + c_{1,0,3}) \notag \\
&= -c_{1,3,0}^2 d^6 + (c_{1,2,1}^2- 2c_{1,3,0}c_{1,1,2})d^4 + (- c_{1,1,2}^2 + 2c_{1,0,3} c_{1,2,1} )d^2 + c_{1,0,3}^2 \notag \\
&= \alpha^3 c_{1,3,0}^2 + (c_{1,2,1}^2- 2c_{1,3,0}c_{1,1,2})\alpha^2 + (c_{1,1,2}^2 - 2c_{1,0,3} c_{1,2,1})\alpha + c_{1,0,3}^2. \label{rhs}
\end{align}

To evaluate the left hand side of \eqref{local_index}, we use the associated trivialization of $E$ as in Lemma \ref{exists}.
Because $(L, Z)$ is a simple zero, $\ind_{(L, Z)} \sigma_f$ is determined by the Jacobian evaluated at $0$ of the induced map from $\aa_K^4 \to \aa_K^4$ given by the section $\sigma_f$.  
With respect to our chosen trivializations, the value in $\aa^4_K$ of this map at the point $(a, b, s, t) \in \aa^4_K$ is the tuple of coefficients expressing $f|_{L_{a,b}}$ modulo $(y_2^2 + s y_2y_3 + (\alpha + t)y_3^2)^2$ as a linear combination of $y_2^3y_3$, $y_2^2y_3^2$, $y_2y_3^3$ and $y_3^4$.  Using \eqref{y2rel}, we have
\begin{align*}
f|_{L_{a,b}} &= (y_2^4 + 2 \alpha y_2^2 y_3^2 + \alpha^2 y_3^4) + (-ay_2 - by_3)( c_{1,3,0} y_2^3 + c_{1,2,1}y_2^2 y_3 + c_{1,1,2} y_2y_3^2 + c_{1,0,3}y_3^3) + O((a,b)^2) \\
 &= 2\alpha y_2^2y_3^2 + \alpha^2y_3^4 + (-ay_2 - by_3) (  c_{1,2,1}y_2^2 y_3 + c_{1,1,2} y_2y_3^2 + c_{1,0,3}y_3^3) - bc_{1,3,0}y_2^3y_3 \\
 &\qquad  + (1 - a c_{1,3,0}) y_2^4 + O((a,b)^2) \\
&=2\alpha y_2^2y_3^2 + \alpha^2y_3^4 + (-ay_2 - by_3) (  c_{1,2,1}y_2^2 y_3 + c_{1,1,2} y_2y_3^2 + c_{1,0,3}y_3^3)  - bc_{1,3,0}y_2^3y_3 \\
&\qquad -(1 - a c_{1,3,0})(2s y_2^3y_3 + 2(\alpha + t) y_2^2y_3^2 + 2s\alpha y_2y_3^2 + (\alpha^2+2\alpha t)y_3^4) + O((a, b, s, t)^2).
\end{align*}
In particular, the Jacobian matrix at $(a, b, s, t) = (0, 0, 0, 0)$ is
\[\left(\begin{matrix}
-c_{1,2,1} & 2c_{1,3,0}\alpha -c_{1,1,2} & -c_{1,0,3} & c_{1,3,0}\alpha^2\\
-c_{1,3,0} & - c_{1,2,1} & -c_{1,1,2} & -c_{1,0,3} \\
-2 & 0 & -2\alpha & 0\\
 0 & -2 & 0 & -2\alpha 
\end{matrix}\right)\]
whose determinant is precisely 4 times \eqref{rhs}.
\end{proof}

\section{Signed counts over $\rr$}

Suppose we have fixed a line $L_\infty$ defined over $\rr$.
Using the relative orientation of Lemma \ref{exists}, Lemma \ref{ouridea} show that the signed count $s(E, \sigma_f)$
is locally constant as a function of $f \in H^0(E)^\circ \smallsetminus V_{D_\infty}$.  
If $Q = V(f)$, Lemma \ref{indlem} then shows that (the geometrically meaningful signed count)
\[
 s_{L_\infty}(Q) \colonequals \# \left({\text{real bitangents with}\atop \text{$\qtype_{L_\infty}(L) = \langle 1 \rangle$}} \right) - \# \left({\text{real bitangents with}\atop \text{$\qtype_{L_\infty}(L) = \langle -1 \rangle$}} \right) 
\]
is constant for $f$ in real connected components of $H^0(E)^\circ \smallsetminus V_{D_\infty}$. 


The following lemma describes a natural pair of connected regions in the space of allowed sections, corresponding quartics whose real points are compact curves in $\aa^2_\rr  = \pp^2_{\rr}\smallsetminus L_\infty$.


\begin{lem} \label{conn}
Let $A$ be the space 
of real quartic polynomials $f$ such that $V(f) \cap L_\infty$ contains no real points and let $A^\circ \subset A$ be those quartics with isolated bitangents.
Then $A^\circ$ is a pair of connected regions inside $H^0(E)^\circ \smallsetminus V_{D_\infty}$, where every section in one connected component $A^\circ$ is the negative of a section in the other component.
\end{lem}
\begin{proof}
Any real bitangent of $V(f)$ meets $L_\infty$ in a real point. 
If $V(f) \cap L_\infty$ contains no real points, it follows that no real bitangent is tangent to $V(f)$ along $L_\infty$. Hence, $A \subset H^0(E)\smallsetminus V_{D_\infty}$.

We first describe the two connected components of $A$. 
Restriction of polynomials to the line at infinity defines a linear map
\[r: H^0(E) \cong H^0(\pp^2, \O_{\pp^2}(4)) \cong \rr^{15} \rightarrow H^0(L_\infty, \O_{L_\infty}(4)) \cong \rr^{5}.\]
Because the defining condition of $A$ depends only on the restriction of polynomials to $L_\infty$, we have $A = r^{-1}(r(A))$. Thus, it suffices to describe $r(A)$.
Give $L_\infty \cong \pp^1$ coordinates $x, y$ so that $H^0(L_\infty, \O_{L_\infty}(4))$ is identified with homogenous degree $4$ polynomials in $x$ and $y$. Consider the map $\rr^5 \rightarrow H^0(L_\infty, \O_{L_\infty}(4))$ defined by
\[(a, b, c, d, e) \mapsto e(x^2 - 2axy + (a^2 + b^2)y^2)(x^2 -2cxy + (c^2 + d^2)y^2).\]
By construction, the image of $\{(a, b, c, d, e) : b, d > 0, e \neq 0\}$ is $r(A)$. Hence, $r(A)$ has two connected components corresponding to the images of the regions for $e > 0$ and $e < 0$, which give quartic polynomials that are negatives of each other.

If $V(f)$ has a positive-dimensional family of bitangents, then $V(f)$ has a non-reduced component; the space of such $f$ is symmetric under negation and occurs codimension greater than $2$, so
$A^\circ$ still consists of two connected components with the described property. 
\end{proof}

\begin{proof}[Proof of Theorem \ref{cap}]
Lemma \ref{indlem} shows that the local index of a bitangent to $V(f)$ is equal to its $\qtype$, which depends only on $f$ up to scaling. Thus, $s(E, \sigma_{f}) = s(E, \sigma_{-f})$, so
Lemma \ref{ouridea} together with Lemma \ref{conn} shows that the signed count is the same for all real quartics not meeting the line at infinity over $\rr$. Thus, it suffices to compute the signed count for a single such quartic.
The Fermat quartic, defined by $f = y_1^4 + y_2^4 + y_3^4$ is smooth, hence has isolated bitangents. Since the curve has no real points, we may choose any line to be the line at infinity. An elementary calculation shows that the Fermat quartic has precisely four real bitangents, defined by
\[V(y_1 + a y_2 + by_3) \qquad \text{for} \qquad a = \pm \sqrt[4]{\frac{\sqrt{5}-1}{2}}, \ b = \pm\sqrt[4]{\frac{\sqrt{5}-1}{3-\sqrt{5}}}\]
Each of these four bitangents meets the curve in a pair of complex conjugate points $p$ and $\overline{p}$, and hence they all have type
\begin{equation} \label{nsp}
\qtype_{L_\infty}(L) = \langle \partial_L f(p) \partial_L f(\overline{p}) \rangle = \langle \partial_L f(p) \overline{\partial_L f(p)} \rangle = \langle 1 \rangle.
\end{equation}
for any line $L_\infty$. Thus, the signed count of bitangents for the Fermat quartic is $4$.
\end{proof}

\subsection{Varying $L_\infty$}\label{vary}
In this section, we explore how the signed count of bitangents to a fixed quartic varies as $L_\infty$ moves.  This is equivalent to studying the signed counts with respect to a fixed line for all of the quartics in a $\PGL_3$ orbit.  

Equation \eqref{nsp} shows that non-split bitangents and hyperflexes have type $\langle 1 \rangle$ with respect to any line at infinity. On the other hand, split bitangents acquire type $\langle -1 \rangle$ and $\langle 1 \rangle$ depending on their geometry relative to the line at infinity. 

Given any plane quartic $Q \subset \pp^2_\rr$, fix a starting line at infinity $M$.
Each split bitangent $L$ to $Q$ determines a line segment $g_M(L)$ in $\aa^2_\rr = \pp^2_\rr \smallsetminus M$ called the \defi{grate of $L$ with respect to $M$}, defined by joining the two points of $L \cap Q$. For any other line $L_\infty$ with $L \cap L_\infty \cap Q = \varnothing$, we have
\[\qtype_{L_\infty}(L) = \begin{cases} \qtype_M(L) & \text{if $L_\infty \cap g_M(L) = \varnothing$} \\ -\qtype_M(L) &\text{if $L_\infty \cap g_M(L) \neq \varnothing$.} \end{cases}\]
It follows that with respect to any line $L_\infty$, the signed count is
\begin{align}\label{differences}
s_{L_\infty}(Q) = s_M(Q) &-2 \cdot \#\{\text{split bitangents $L$}:\text{$\qtype_{M}(L) = 1$ and $L_\infty \cap g_M(L) \neq \varnothing$}\} \\
&+ 2 \cdot \#\{\text{split bitangents $L$}:\text{$\qtype_{M}(L) = -1$ and $L_\infty \cap g_M(L) \neq \varnothing$}\}. \notag
\end{align}
Given a quartic $Q$, we determine all possible signed counts using the following algorithm.


\vspace{.1in}
\noindent
\begin{algo}\label{algo:all_counts} Input: equation $f$ of a smooth plane quartic. Output: set of all possible signed counts of bitangents.
\begin{enumerate}
\item Compute equations defining the bitangents to $V(f)$ by elimination. Find endpoints of grates of all split bitangents.

\item Choose any starting line $M = V(z)$ which does not meet the curve at the endpoint of any grate. Compute the $\qtype$ of each bitangent with respect to $M$.

\item The endpoints of the grate of a split bitangent defines a pair of lines in ${\pp^2}^\vee$ (intersecting in the point of ${\pp^2}^\vee$ corresponding to the bitangent). Call the collection of all such lines the \defi{dual grate arrangement} (pictured in solid blue below).  The signed count is constant for $[L_\infty]$ in the complement of the dual grate arrangement. Thus, it suffices to check the signed count with respect to a representative in each region. The vertices of dual grate arrangement (black dots below) correspond to lines joining the endpoints of grates.

\item We sample the finitely many regions of the complement of the dual grate arrangement with red test lines in ${\pp^2}^\vee$ through $[M]$ as pictured below. It suffices to check a finite collection of red lines, one in each region bounded by the dashed black lines joining $[M]$ and vertices of the dual grate arrangement.

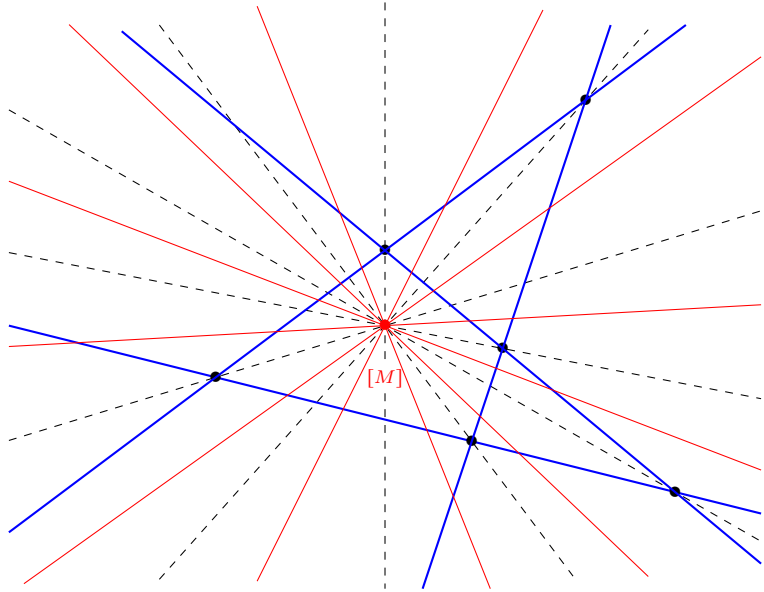
\begin{figure}
\begin{center}
\begin{tikzpicture}
\draw [dashed] (0, 4.3) -- (0, -.5);
\draw [dashed] (0, -.9) -- (0, -3.5);
\draw node at (0, 1) {$\bullet$};
 \draw[domain=-5:5,smooth,variable=\x,blue, thick] plot ({\x},{-\x/4 - 2.5/2});
 \draw node at (-2.25, 2.25/4 - 2.5/2) {$\bullet$};
 \draw[domain=-5:5,smooth,variable=\x,black,dashed] plot ({\x},{\x*(2.25/4 - 2.5/2)/(-2.25)});
\draw node at ( 3.85714285714286, -2.21428571428571) {$\bullet$};
 \draw[domain=-5:5,smooth,variable=\x,black,dashed] plot ({\x},{\x*-2.21428571428571/3.85714285714286});
\draw node at (1.15384615384615, -1.53846153846154) {$\bullet$};
 \draw[domain=-3:2.5,smooth,variable=\x,black,dashed] plot ({\x},{\x*-1.53846153846154/1.15384615384615});
\draw node at (8/3, 3) {$\bullet$};
 \draw[domain=-3:3.5,smooth,variable=\x,black,dashed] plot ({\x},{\x*3/(8/3)});
\draw node at (36/23, -7/23) {$\bullet$};
\draw [domain=-5:5,smooth,variable=\x,black,dashed] plot ({\x},{\x* -(7/23)/(36/23)});
  \draw[domain=-5:4,smooth,variable=\x,blue, thick] plot ({\x},{3*\x/4 + 1});
  \draw[domain=-3.5:5,smooth,variable=\x,blue, thick] plot ({\x},{-5/6*\x +1});
 \draw[domain=.5:3,smooth,variable=\x,blue, thick] plot ({\x},{3*\x-5});
 \draw [red] node (0, 0) {$\bullet$};
 \draw [red] node at (0, -.7) {\tiny $[M]$};
  \draw[domain=-4.2:3.5,smooth,variable=\x,red] plot ({\x},{-0.953703703703706*\x});
    \draw[domain=-5:5,smooth,variable=\x,red] plot ({\x},{-0.384259259259258*\x});
\draw[domain=-5:5,smooth,variable=\x,red] plot ({\x},{0.0555555555555556*\x});
  \draw[domain=-4.8:5,smooth,variable=\x,red] plot ({\x},{0.715277777777778*\x});
\draw[domain=-1.7:2.1,smooth,variable=\x,red] plot ({\x},{2*\x});
\draw[domain=-1.7:1.4,smooth,variable=\x,red] plot ({\x},{-2.5*\x});
\end{tikzpicture}
\caption{The dual grate arrangement and test pencils in ${\pp^2}^\vee$}
\end{center}
\end{figure}

If $M = V(z) \subset \pp^2$, then a red line in ${\pp^2}^\vee$ represents a family of parallel lines with fixed slope in the $(x, y)$ plane.   Thus it suffices to check all lines of slope $a$ for some finite and computable set of $a$.



\item As $L_\infty$ moves along lines of slope $a$, the signed count only changes when $L_\infty$ meets the endpoint of a grate (when a red line crosses a blue line above). The order that different endpoints in $\pp^2$ are hit is determined by their projection onto a line of slope $-1/a$. Sort the list of all endpoints accordingly, together with the effect they will have when crossed. The partial sums of the effects in this sorted list determine all possible signed counts attained in this pencil as in \eqref{differences}.
\end{enumerate}
\end{algo}

An implementation in Sage is available at \cite{code}.

\subsection{Example: the Trott curve}\label{trott}

The Trott curve $Q$ is given by the vanishing of the homogeneous quartic polynomial
$$f = 12^2 (x^4 + y^4) - 15^2(x^2 + y^2)z^2 + 350x^2y^2 + 81z^4.$$
Topologically, the real points of this quartic form $4$ non-nested ovals, and all $28$ bitangents are defined over $\rr$ and split.  Therefore the sign of every bitangent line depends on the choice of $L_\infty$.

Carrying out Algorithm \ref{algo:all_counts} with starting line $M = V(z)$ verifies Conjecture \ref{conj_range} for this quartic: as $L_\infty$ ranges over all real lines in $\pp^2$, $s_{L_\infty}(Q)$ always lies in the set $\{0,2,4,6,8\}$. 
Furthermore, every such count is achieved for some choice of $L_\infty$.  The colored band in Figure 2 below indicates the possible signed counts for lines $L_\infty$ in the pencil of slope $5/4$.  

\begin{figure}[ht]\label{fig:band}
\begin{center}
\begin{minipage}[l]{0.45\textwidth}
\includegraphics[width=3in]{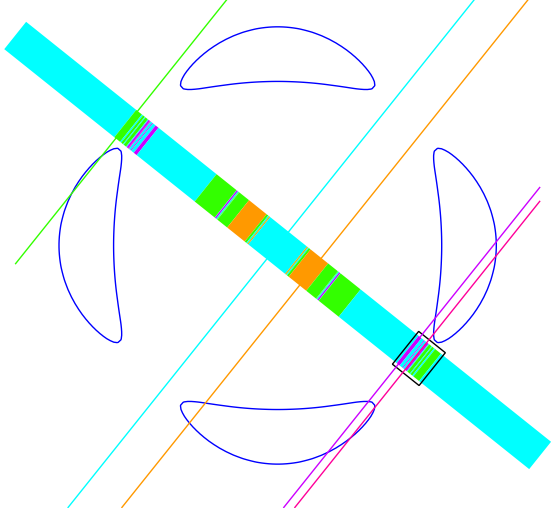}
\end{minipage}
\begin{minipage}[l]{0.31\textwidth}
\includegraphics[width=1.9in]{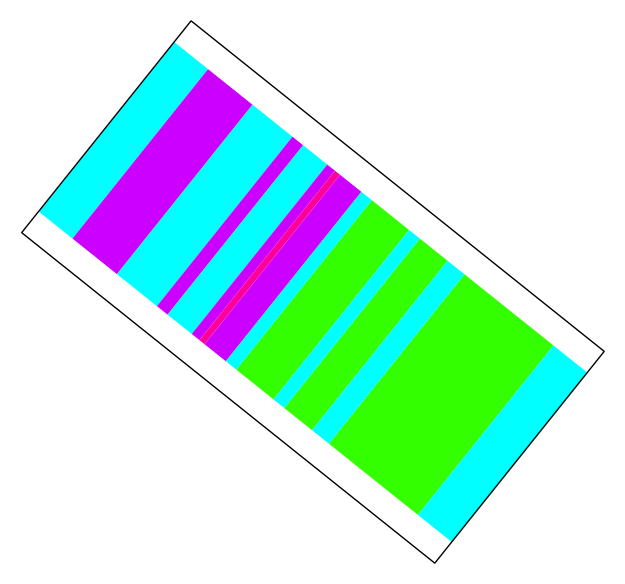}
\end{minipage}
\begin{minipage}[l]{0.03\textwidth}
\hspace{10pt}
\includegraphics[width=.2in]{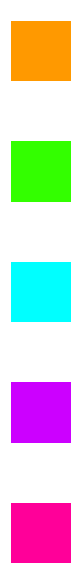}
\end{minipage}
\begin{minipage}[l]{0.18\textwidth}
$s_{L_\infty}=0$\\[8pt]
$s_{L_\infty}=2$ \\[8pt]
$s_{L_\infty}=4$ \\[8pt]
$s_{L_\infty}=6$\\[8pt]
$s_{L_\infty}=8$
\end{minipage}
\caption{The possible signed counts of bitangents of the Trott curve with respect to a line at infinity $L_\infty$ varying in the pencil of lines of slope $5/4$ 
}
\end{center}
\end{figure}

In Figure 3 we illustrate $5$ lines $L_0, L_2, L_4, L_6, L_8$ from this pencil that achieve each of the five possible signed counts. The figure shows the grates with respect to $M$ that intersect each $L_i$, and hence change sign with respect to $L_i$. Black indicates that $\qtype_M(L_i)=1$ and red indicates that $\qtype_M(L_i)=-1$.

\begin{figure}[ht]
\begin{center}
\begin{minipage}[l]{0.31\textwidth}
\begin{center}
$L_2 = V(y - 1.25x - 1.415)$\\
\includegraphics[width=2in]{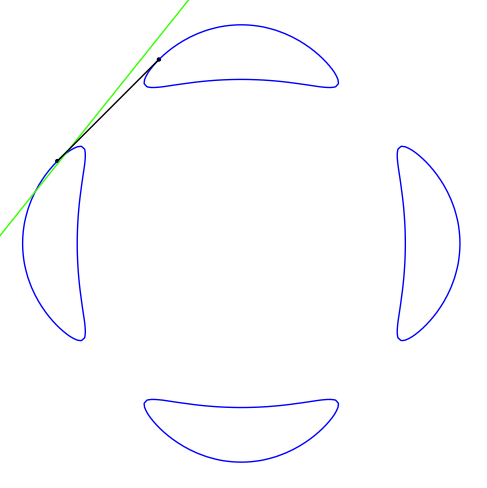}
\end{center}
\end{minipage}
\begin{minipage}[l]{0.31\textwidth}
\begin{center}
$L_4 = V(y - 1.25x)$\\
\includegraphics[width=2in]{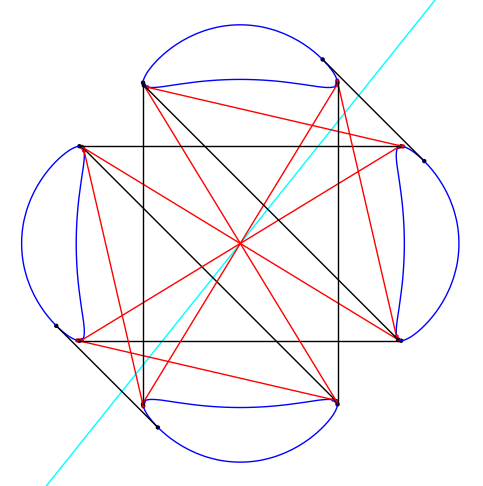}
\end{center}
\end{minipage}
\end{center}

\vspace{10pt}

\begin{minipage}[l]{0.31\textwidth}
\begin{center}
$L_0 = V(y - 1.25x + 0.3075)$\\
\includegraphics[width=2in]{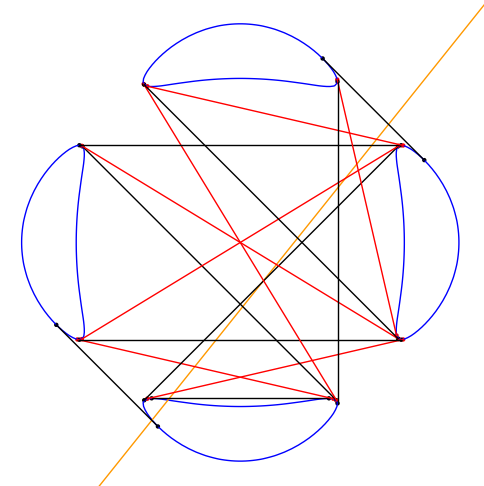}
\end{center}
\end{minipage}
\begin{minipage}[l]{0.31\textwidth}
\begin{center}
$L_6 = V(y-1.25x + 1.233)$\\
\includegraphics[width=2in]{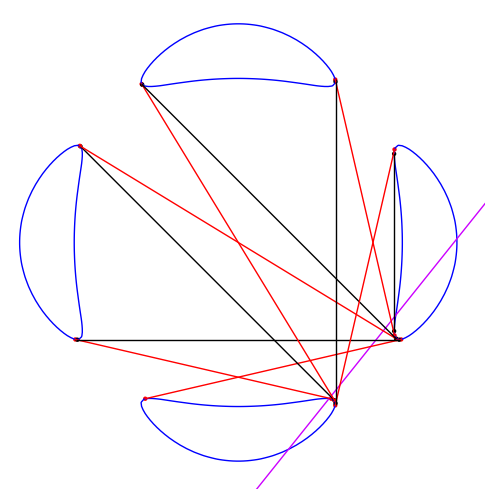}
\end{center}
\end{minipage}
\begin{minipage}[l]{0.31\textwidth}
\begin{center}
$L_8 = V(y - 1.25x + 1.296)$\\
\includegraphics[width=2in]{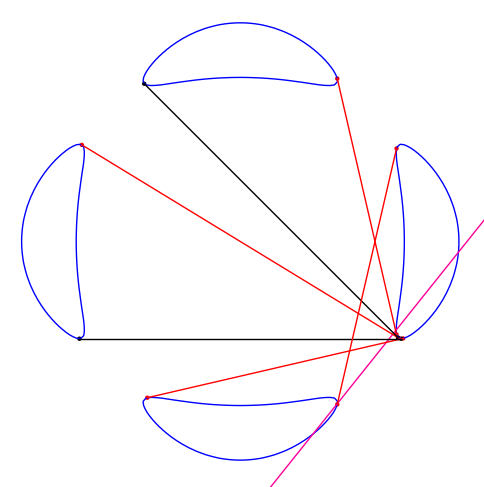}
\end{center}
\end{minipage}
\caption{Choices for $L_\infty$ achieving each of the possible signed counts}
\end{figure}


\section{Comparison with lines on smooth cubics}
In the previous sections, we gave a procedure for computing the local index relative to a line $L_\infty$ and interpreted it as a geometric type in terms of the local geometry of the quartic.
We now relate the relative type of a bitangent to the type of a line on a cubic surface.

\begin{defin}
Let $Q$ be a smooth plane quartic and let $L_\infty$ be bitangent line defined over $k$.
We say that a pointed cubic $(V, p)$ in $\pp^3_k$ is \defi{associated to $(Q, L_\infty)$} if
the projection map from $p$
\[\pi_p \colon V \dashrightarrow \pp_k^2 \]
has branch divisor $Q$ and $\pi_p(T_pV \cap V) = L_\infty$.
\end{defin}

\begin{lem} \label{existscubic}
Let $Q \subset \pp^2_k$ be a smooth plane quartic and let $L_\infty \subset \pp^2_k$ be a bitangent of $Q$ defined over $k$.  Then there exists an associated cubic $(V, p)$ defined over $k$.
\end{lem}
\begin{proof}
Choose a homogeneous polynomial $f(x,y,z)$ of degree $4$ defining $Q$ such that $f|_{L_\infty}$ is a square (this is well-defined up to multiplication by elements of $(k^*)^2$).  Let $\tilde{V}$ be the double cover of $\pp^2_k$ specified in weighted projective space $\pp(2,1,1,1)_{wxyz}$ by the equation
\begin{equation}\label{doublecover} w^2 = f(x,y,z). \end{equation}
The surface $\tilde{V}$ is a del Pezzo of degree $2$, and the double cover map $\phi _{-K_{\tilde{V}}} \colon \tilde{V} \to \pp^2_k$ is induced by the complete linear system $|-K_{\tilde{V}}|$.
As $f|_{L_\infty}$ is a square, the preimage of $L_\infty$ in $\tilde{V}$ splits as two $k$-divisors $E_1$ and $E_2$ intersecting above the points at which $L_\infty$ is tangent to $Q$.  For either choice of $i=1$ or $2$, the image of $(\tilde{V}, E_i)$ under the linear system $|-K_{\tilde{V}} + E_i|$ is a pointed smooth cubic surface $(V, p)$ in $\pp^3_k = \pp H^0(\tilde{V}, -K_{\tilde{V}} + E_i)^\vee.$  The subspace $H^0(\tilde{V}, -K_{\tilde{V}}) \subseteq H^0(\tilde{V}, -K_{\tilde{V}} + E_i)$ induces a linear map $\pp H^0(\tilde{V}, -K_{\tilde{V}} + E_i)^\vee \dashrightarrow \pp H^0(\tilde{V}, -K_{\tilde{V}})^\vee$, which is projection from the $1$-dimensional quotient $H^0(E_i, (-K_{\tilde{V}} + E_i)|_{E_i}) \simeq H^0(E_i, \O_{E_i})$, i.e., the point $p = \phi _{-K_{\tilde{V} + E_i}}(E_i)$.
In other words, the composite map
\[\tilde{V} \to \pp H^0(\tilde{V}, -K_{\tilde{V}} + E_i)^\vee \dashrightarrow \pp H^0(\tilde{V}, -K_{\tilde{V}})^\vee \]
is both the original double cover map $\phi _{-K_{\tilde{V}}}$ away from $E_i$, and projection of $\phi _{-K_{\tilde{V} + E_i}}(\tilde{V})$ from $p \in \pp H^0(\tilde{V}, -K_{\tilde{V}} + E_i)^\vee$.  Hence the branch divisor is the quartic curve $Q$, as desired.  Furthermore, the image of $E_j$ for $j \neq i$ in $\pp H^0(\tilde{V}, -K_{\tilde{V}} + E_i)^\vee$ is a curve of degree $3 = (-K_{\tilde{V}} + E_i) \cdot E_j$ with multiplicity $2 = E_i \cdot E_j$ at $p$.  Any such curve is necessarily planar, and therefore $\phi _{-K_{\tilde{V} + E_i}} (E_j)$ must be the tangent plane section $T_pV \cap V$.  The image of $E_j$ in $\pp H^0(\tilde{V}, -K_{\tilde{V}})^\vee \simeq \pp^2_k$ is evidentally the bitangent $L_\infty$.
\end{proof}

\begin{rem}
In the proof of Lemma \ref{existscubic}, we chose the unique twist of \eqref{doublecover} such that the preimage of $L_\infty$ on the double cover splits into two exceptional curves.  This is essential so that we may blow down just one of them to obtain a cubic surface.
\end{rem}

 Each of the remaining $27$ bitangent lines to $Q$ corresponds to a unique line on a cubic surface $V$, which therefore has the same field of definition. We will show that the type of the bitangent line relative to $L_\infty$ is equal to the type of the corresponding line on an associated cubic surface. Recall that if $L \subset V$ is a line on a cubic surface, then projection from $L$ restricts to a degree $2$ cover $L \rightarrow \pp^1$, whose associated involution we denote $\iota: L \rightarrow L$. Kass--Wickelgren define the \defi{type} of $L \subset V$, denoted $\type(L)$, to be the class in $\GW(k(L))$ of the discriminant of the fixed locus of $\iota$.
 
 \begin{lem} \label{sametype}
Suppose that $Q$ is a smooth plane quartic and $L_\infty$ is a bitangent to $Q$ defined over $k$. Let $(V, p)$ be an associated cubic. For each bitangent $L \neq L_\infty$ to $Q$ defined over $K$, we have
\[\qtype_{L_\infty}(L) = \type(\tilde{L}) \in \GW(K),\]
where $\tilde{L} \subset V$ is the unique line such that $\pi_p(\tilde{L}) = L$.
\end{lem}
\begin{proof}
Given a pointed cubic surface $p \in V = V(F)$, one can recover the equation of the corresponding quartic explicitly, allowing us to relate our two notions of type. Fix some $\tilde{L} \subset V$. Our assumption that $L \neq L_\infty$ means that $p \notin \tilde{L}$.
We may choose coordinates $[x_0, x_1, x_2, x_3]$ on $\pp_K^3$ so that
\begin{equation} \label{conditions}
p = [1, 0, 0, 0], \qquad T_p V = V(x_3), \qquad \text{and} \qquad L = V(x_0, x_1).
\end{equation}
With respect to these coordinates, our cubic equation has the form
\[F = \sum_{i+j+k+\ell = 3} a_{i,j,k,l} x_0^i x_1^j x_2^k x_3^\ell,\]
and the conditions in \eqref{conditions} imply
\[a_{3,0,0,0} = a_{2,1,0,0} = a_{2,0,1,0} = a_{0,0,3,0} = a_{0,0,2,1} = a_{0,0,1,2} = a_{0,0,0,3} = 0.\]
By \cite[Lemma 50]{kass_wickelgren}, the type of the line $\tilde{L}$ is $\langle M \rangle$ where 
\[M = \det \left(\begin{matrix} a_{1,0,2,0} & 0 & a_{0, 1, 2, 0} & 0 \\
a_{1, 0, 1, 1} & a_{1, 0, 2, 0} & a_{0, 1, 1, 1} & a_{0, 1, 2, 0} \\
a_{1, 0, 0, 2} & a_{1, 0, 1, 1} & a_{0, 1, 0, 2} & a_{0, 1, 1, 1} \\
0 & a_{1, 0, 0, 2} & 0 & a_{0, 1, 0, 2}
 \end{matrix} \right). \]
The lines through $p$ in $\pp^3$ are parametrized by a $\pp^2$ with coordinates $[y_1, y_2, y_3]$ where
\[[y_1, y_2, y_3] \leftrightarrow \{[s, ty_1, ty_2, ty_3] : [s, t] \in \pp^1\}.\]
The restriction of $F$ to one of these lines is given by
\[\sum_{i+j+k+\ell = 3} a_{i,j,k,l} s^i (ty_1)^j (ty_2)^k (ty_3)^\ell = t(s^2 A + stB + t^2 C),\]
where 
\begin{align*}
A &= a_{2,0,0,1}y_3\\
B &= a_{1,2,0,0}y_1^2 + a_{1,1,1,0}y_1y_2 + a_{1,1,0,1}y_1y_3 + a_{1,0,2,0}y_2^2+ a_{1,0,1,1}y_2y_3 + a_{1,0,0,2}y_3^2 \\
C &= a_{0,1,2,0}y_1y_2^2 + a_{0,1,1,1}y_1y_2y_3 + a_{0,1,0,2}y_1y_3^2 + a_{0,2,1,0}y_1^2y_2 + a_{0,2,0,1}y_1^2y_3 + a_{0,3,0,0}y_1^3.
\end{align*}
The branch divisor on $\pp^2$ is the locus where the residual quadratic $s^2 A + stB + t^2 C$ has a double root. Thus, the quartic is given by the vanishing of the equation $f = B^2 - 4AC$. The image of $\tilde{L} \subset V$ is the line $V(y_1) \subset \pp^2$,
which one readily checks is a bitangent to $V(f) \subset \pp^2$. Indeed, substituting $y_1 = 0$ into $f$ gives the quartic $(a_{1,0,2,0}y_2^2 + a_{1,0,1,1}y_2y_3 + a_{1,0,0,2}y_3^2)^2$. Thus, the tangency subscheme of $V(f)$ along $V(y_1)$ is $z_1+ z_2$ where
\[z_1=[0, -a_{1,0,1,1} + d, 2a_{1,0,2,0}] \qquad \text{and} \qquad z_2=[0, -a_{1,0,1,1} - d, 2a_{1,0,2,0}] \]
with $d^2 = a_{1,0,1,1}^2 - 4a_{1,0,2,0}a_{1,0,0,2}$. Explicit computation shows that
\[\frac{\partial f}{\partial y_1}(z_1) \cdot \frac{\partial f}{\partial y_1}(z_2) = 1024 a_{2,0,01}^2 a_{1, 0, 2, 0}^4 \cdot M.\]
Since the two quantities differ by a square, they are equal in the Grothendieck-Witt group of $K$. 
In other words,
\[\qtype_{L_\infty}(L) = \qtype_{\pi_p(T_pV)}(\pi_p(\tilde{L})) = \type(\tilde{L}).\qedhere\]
\end{proof}

Recall that the $\qtype$ of a line with residue field a non-trivial extension of $k$ is defined to be the $\qtype$ of some representative line defined over $k(L)$.
\begin{cor} \label{tr}
For any bitangent line $L$ of $(Q, L_\infty)$ with associated cubic $(V, p)$, we have
\[\Tr_{k(L)/k} \qtype_{L_\infty}(L) = \Tr_{k(\tilde{L})/k} \type(\tilde{L})\]
where $\tilde{L}$ is the unique line on $V$ such that $\pi_p(\tilde{L}) = L$.
\end{cor}

\begin{proof}[Proof of Theorem \ref{main}]
Let $(V,p)$ be a pointed cubic associated to $(Q, L_\infty)$. Summing over bitangents to $Q$ and applying Corollary \ref{tr}, the main theorem of Kass--Wickelgren \cite[Thm. 2]{kass_wickelgren} now shows
\begin{align*}
\sum_{\substack{\text{lines } L \text{ bitangent to } Q \\ L \neq L_\infty}} \Tr_{k(L)/k} (\qtype_{L_\infty}(L)) &= \sum_{\text{lines } \tilde{L} \subset V} \Tr_{k(L)/k} \type(\tilde{L}) \\
&=15\langle 1 \rangle + 12\langle -1 \rangle \in \GW(k). \qedhere
\end{align*}
\end{proof}

\bibliographystyle{plain}
\bibliography{bitangents}

\end{document}